\documentclass[11pt, a4paper]{amsart}
\usepackage{amsmath,amssymb,amsthm,mathtools,wasysym,calc,verbatim,tikz,url,hyperref,mathrsfs,cite,fullpage,bbm,tikz-cd}
\usetikzlibrary{shapes,arrows}
\usetikzlibrary{decorations.pathreplacing}
\usetikzlibrary{3d}
\usetikzlibrary{intersections}
\usepackage[noabbrev,capitalize,nameinlink]{cleveref}
\usepackage[shortlabels]{enumitem}
\usepackage{thm-restate}

\usepackage{comment}

\numberwithin{equation}{section}
\newtheorem{theorem}{Theorem}[section]
\newtheorem{lemma}[theorem]{Lemma}

\newtheorem{prop}[theorem]{Proposition}

\newtheorem{claim}[theorem]{Claim}

\newtheorem{problem}[theorem]{Problem}

\theoremstyle{definition}

\theoremstyle{remark}

\newtheorem*{remark*}{Remark}

\begin{document}

\title{A new proof of two-color partition regularity of Pythagorean triples}
\author{James Leng}

\begin{abstract}
We give a new proof that if the natural numbers are colored in two colors, then there exists a monochromatic Pythagorean triple. The proof was found by ChatGPT Astra.    
\end{abstract}

\maketitle

\section{Introduction}
In \cite{Graham-prob-1, Graham-prob-2}, Erd\"os and Graham pose the following notorious problem.
\begin{problem}
Let $c: \mathbb{N} \to \{1, \dots, K\}$ be a coloring of the natural numbers. Do there exist positive integers $x,y,z$ such that $x^2+y^2=z^2$ and $c(x)=c(y)=c(z)$?
\end{problem}
In other words, are the Pythagorean triples \emph{partition regular}? The precise origin of this problem appears to be somewhat uncertain. Graham later recalled that he and Erd\"os had discussed the problem in talks for some time, but that he did not remember when they first published it. In a later survey, Graham described the problem as having been open for more than thirty years and cited the 1980 Erd\"os--Graham monograph. Thus the problem was certainly circulating by the mid-1970s and was in print by 1980.

The general finite-color problem remains open. The case $K=2$ was answered in the affirmative by Heule, Kullmann, and Marek \cite{heule-kullmann-marek} using a computer-assisted SAT proof. Their formal DRAT proof was almost $200$ terabytes in size; they also produced a compressed certificate of $68$ gigabytes from which the proof could be reconstructed and checked. In this article we give a shorter analytic proof of the two-color case.
\begin{theorem}\label{thm:main}
Let $c: \mathbb{N} \to \{-1,1\}$ be a coloring. Then there exist positive integers $x,y,z$ such that $x^2+y^2=z^2$ and $c(x)=c(y)=c(z)$.
\end{theorem}
We use the colors $\{-1,1\}$ rather than $\{1,2\}$ because the proof relies on the following elementary identity: if $c$ has no monochromatic Pythagorean triple and $a^2+b^2=d^2$, then
$$c(a)c(b)+c(a)c(d)+c(b)c(d)=-1.$$
Such a criterion is quite standard (though novel in this specific context) and has appeared in 3SAT contexts \cite{ODonnell2014}\footnote{We thank Mehtaab Sawhney for pointing this out.}. Apparently it has been known since the 1920s (e.g., \cite{Ising1925}) in the context of the Ising model. 
We briefly recall some related work. Frantzikinakis and Host \cite{Fra-Host-structure-multiplicative} proved some of the first partition-regularity results for homogeneous quadratic equations in three variables. For example, for every finite coloring of $\mathbb{N}$ there exist distinct monochromatic $x,y$ and some $z\in\mathbb{N}$ such that
$$16x^2+9y^2=z^2.$$
Building on this work, and also introducing additional tools via the study of pretentious and non-pretentious multiplicative functions, Frantzikinakis, Klurman, and Moreira \cite{Fra-Klu-Mor} proved partition regularity and multiplicative-density regularity for Pythagorean pairs. In particular, their results guarantee monochromatic pairs of coordinates lying in a Pythagorean triple, and give corresponding multiplicative density statements. Our proof makes essential use of tools developed in \cite{Fra-Host-structure-multiplicative, Fra-Klu-Mor}. In fact, our methods are closer in spirit to those developed in Frantzikinakis and Mountakis \cite{pret_pyth}, who prove multiple-recurrence results for pretentious multiplicative systems along generalized Pythagorean triples.

The methods of \cite{Fra-Host-structure-multiplicative, Fra-Klu-Mor} proceed by applying a Furstenberg correspondence principle to indicator functions of sets. This direction is difficult to make work because triple correlations
$$\int F \circ T_a F \circ T_b F \circ T_d d\mu$$
are not well understood. Our argument is slightly different: we apply the correspondence principle directly to the coloring. Let $\Psi=(\Psi_N)_{N\in\mathbb{N}}$ be a multiplicative F\o lner sequence for $\mathbb{Q}^{\times}_{>0}$ with $\Psi_N\subseteq\mathbb{N}$. After passing to a subsequence $N_k$, we obtain a probability space $\mathbf{X}=(X,\mathcal{X},\mu)$, a measure-preserving $\mathbb{Q}^{\times}_{>0}$-action $(T_q)_{q\in\mathbb{Q}^{\times}_{>0}}$, and a function $F\in L^\infty(\mathbf{X})$ taking values in $\{-1,1\}$ such that
\begin{equation}\label{eq:Furstenberg-Correspondence}
\int_X F(T_a x)F(T_bx)\,d\mu(x)
=
\lim_{k\to\infty}\mathbb{E}_{n\in\Psi_{N_k}}c(an)c(bn)
\end{equation}
for all $a,b\in\mathbb{N}$.

For completeness sake, we give an explicit construction. Take the shift space $X=\{-1,1\}^{\mathbb{Q}_{>0}^{\times}}$ with
$$
(T_qx)_r=x_{rq},
$$
regard $c$ as a point of $X$ (extending it arbitrarily off $\mathbb{N}$), and set $F(x)=x_1$. The measure $\mu$ is a weak-$*$ limit of
$$
\mathbb{E}_{n\in\Psi_{N_k}}\delta_{T_nc}.
$$
The F\o lner property implies that $\mu$ is invariant, and \eqref{eq:Furstenberg-Correspondence} follows immediately from the definition of $F$.

Now suppose that there is no monochromatic Pythagorean triple. If $a^2+b^2=d^2$, then for every $n\in\mathbb{N}$,
$$
c(an)c(bn)+c(an)c(dn)+c(bn)c(dn)=-1.
$$
Averaging and applying \eqref{eq:Furstenberg-Correspondence} gives
$$
\langle T_aF,T_bF\rangle+\langle T_aF,T_dF\rangle+\langle T_bF,T_dF\rangle=-1.
$$
Since $F$ takes values in $\{-1,1\}$, the integrand
$$
F(T_ax)F(T_bx)+F(T_ax)F(T_dx)+F(T_bx)F(T_dx)
$$
takes values in $\{-1,3\}$. Its integral is $-1$, so it equals $-1$ almost everywhere. Since there are only countably many Pythagorean triples, we may pass to an ergodic component on which this identity holds simultaneously for every such triple\footnote{This is not strictly necessary, though we do so anyway since it appears that as written the factors $\mathcal{X}_{\text{p. fin. supp}}$ as defined in \cite{pret_pyth} are only proven to be factors in the case where the underlying transformation is ergodic, though the result should hold for a general probability preserving system.}. We henceforth assume that $T$ is ergodic. The remainder of the proof analyzes the consequences of this identity. In contrast with the other method discussed, this method only relies on double correlations
$$\int F \circ T_a F \circ T_b d\mu$$
which are well-understood by the aforementioned works \cite{Fra-Host-structure-multiplicative, Fra-Klu-Mor, pret_pyth}. (In fact, understanding these double correlations is what allowed the authors of \cite{Fra-Host-structure-multiplicative, Fra-Klu-Mor} to prove their pythagorean pairs results.)

\subsection{Acknowledgements}
We thank Mehtaab Sawhney for pointing out the connection to \cite{ODonnell2014}. This proof was found by ChatGPT. For more information on how this proof was found, we refer the reader to Section~\ref{sec:aiacknowledgement}.

\section{Notation and conventions}
We shall write $\mathbb{E}_{n \in A} = \frac{1}{|A|} \sum_{n \in A}$. \emph{Multiplicative functions} will refer to \emph{completely multiplicative functions}. We say that $f = O(g)$ or $f \ll g$ if there exists a constant $C$ such that $|f| \le C|g|$.

Put $P_1(m,n)=2mn$, $P_2(m,n)=m^2-n^2$, and $P_3(m,n)=m^2+n^2$. Fix integers
\[
5\le r<K<Y
\]
and define
\[
\Phi_r
:=
\left\{
\prod_{p\le r}p^{e_p}:r<e_p\le \frac{3r}{2}
\right\},
\]
\[
\Phi_{r,K}
:=
\left\{
\prod_{r<p\le K}p^{e_p}:K<e_p\le \frac{3K}{2}
\right\}.
\]
Set
\[
Q=Q_Y:=\prod_{p\le Y}p^{4Y}.
\]
Choose independently
\[
q_1\in\Phi_r,
\qquad q_2\in\Phi_{r,K}.
\]

\begin{lemma}\label{lem:CRT}
For each such $q_1,q_2$ one may choose $0\le v<Q$ so that
\begin{equation}\label{eq:gcd-normalization}
(2v,Q)=q_1,
\qquad
(1-v^2,Q)=q_2,
\qquad
(1+v^2,Q)=1,
\end{equation}
and, in addition,
\begin{equation}\label{eq:crt-product-congruences}
\frac{2v}{q_1}\equiv1\pmod{q_1},
\qquad
\frac{1-v^2}{q_2}\equiv q_2^{-1}\pmod{q_1},
\qquad
1+v^2\equiv1\pmod{q_1}.
\end{equation}
\end{lemma}
This is a close variant of the modular construction used by Frantzikinakis and Mountakis \cite{pret_pyth}. 
\begin{proof}
By the Chinese remainder theorem, it suffices to specify $v$ modulo the relevant prime powers dividing $Q$. Write $q_1 = \prod_{p \le r} p^{e_p}$ and $q_2 = \prod_{r < p \le K} p^{d_p}$. For $p \le r$, we put $v \equiv \frac{q_1}{2} \pmod{q_1^2/2}$. Hence, $\frac{2v}{q_1} \equiv 1 \pmod{q_1}$ and $(2v, Q) = q_1$. Moreover, we have
$$v^2 \equiv 0 \pmod{q_1} \implies 1 + v^2 \equiv 1 \pmod{q_1}.$$
For $r < p \le K$, we let $v \equiv 1 + p^{d_p} \pmod{p^{d_p + 1}}$. Since $p$ is odd, $v+1\equiv2\pmod p$, and hence $\nu_p(1-v^2)=\nu_p(1-v)=d_p$. Also, $(p,2v)=1$ and $(p,v^2+1)=1$. Finally, for $K < p \le Y$, simply choose $v$ to lie away from the roots of $P(x) = x(x^2 - 1)(x^2 + 1)$. The Chinese remainder theorem now gives $v\pmod Q$ satisfying all these local conditions, and these conditions imply the asserted gcd identities and congruences.
\end{proof}
Now let
\begin{equation}\label{eq:affine-grid}
M=Qm+1,
\qquad
N=Qn+v.
\end{equation}
For $L\ge2$ define
\begin{equation}\label{eq:angular-weight}
w_L(m,n)
:=
\frac mn\,
\mathbf{1}_{\{L\le\log(m/n)\le2L\}}
\end{equation}.
Let
$$\mathbb{E}_{\substack{m, n \le X \\ m, n \sim w_{L}}} f(m, n) := \frac{\sum_{m, n \le X} w_L(m, n) f(m, n)}{\sum_{m, n \le X} w_L(m, n)}.$$
We shall work with the averages
$$\mathbb{E}_{q_1 \in \Phi_r}\mathbb{E}_{q_2 \in \Phi_{r, K}}\mathbb{E}_{\substack{m, n \le X \\ m, n \sim w_L}} f(P_i(M, N)/P_j(M, N)).$$
For the iterated limsups below, the parameters are sent to infinity in the order
\begin{equation}\label{eq:limit-order}
X\to\infty,
\qquad
Y\to\infty,
\qquad
K\to\infty,
\qquad
r\to\infty,
\qquad
L\to\infty.
\end{equation}
We denote the corresponding nested limsup by $\limsup_{L,r,K,Y,X}$. 

\section{Reducing to an ergodic-theoretic result}
Our main ergodic theoretic result is the following.
\begin{theorem}\label{thm:maindecomposition}
For every $f\in L^\infty(\mathbf{X})$ and every $1\le i<j\le3$,
$$\limsup_{L,r,K,Y,X} \left\|\mathbb{E}_{q_1 \in \Phi_r}\mathbb{E}_{q_2 \in \Phi_{r,K}}\mathbb{E}_{\substack{m,n\le X \\ m, n \sim w_L}} T_{P_i(M,N)/P_j(M,N)}f - \mathcal{P}_{ij}f\right\|_{L^2(\mathbf{X})} = 0.$$
where
$$\mathcal{P}_{ij}f = \begin{cases} \left(\int_X f\,d\mu\right)\mathbf{1} & (i,j) \in \{(1,2),(1,3)\}, \\ \mathbb{E}(f\mid \mathcal{X}_{\text{p. fin. supp.}}) & (i,j)=(2,3).\end{cases}$$
\end{theorem}
Here $\mathcal{X}_{\text{p. fin. supp.}}$ is the factor defined in Frantzikinakis--Mountakis \cite{pret_pyth}: its $L^2$-space consists of functions whose spectral measures are supported on pretentious multiplicative functions $g$ for which $g(p)=\chi(p)p^{it}$ for all but finitely many primes $p$, for some Dirichlet character $\chi$ and $t\in\mathbb{R}$.

Let us prove Theorem~\ref{thm:main} assuming Theorem~\ref{thm:maindecomposition}.
\begin{proof}[Proof of Theorem~\ref{thm:main} assuming Theorem~\ref{thm:maindecomposition}.]
As before, assume that the coloring $c$ has no monochromatic Pythagorean triple. We now take $a = M^2 - N^2$, $b = 2MN$, and $d = M^2 + N^2$. Averaging the identity
$$\langle T_aF,T_bF\rangle+\langle T_aF,T_dF\rangle+\langle T_bF,T_dF\rangle=-1$$
and using $\langle T_\ell F,T_mF\rangle=\langle T_{\ell/m}F,F\rangle$, we may choose a sequence $(L_s,r_s,K_s,Y_s,X_s)$ along which the three errors in Theorem~\ref{thm:maindecomposition} tend to zero. Along this sequence,
\begin{align*}
-1 &= \lim_{s \to \infty} \mathbb{E}_{q_1 \in \Phi_{r_s}}\mathbb{E}_{q_2 \in \Phi_{r_s,K_s}}\mathbb{E}_{\substack{m,n \le X_s \\ m, n \sim w_{L_s}}}\left(\langle T_aF, T_bF\rangle + \langle T_aF, T_dF \rangle + \langle T_bF, T_dF \rangle\right) \\
&= 2\left|\int_X F\,d\mu\right|^2+\|\mathbb{E}(F\mid\mathcal{X}_{\text{p. fin. supp.}})\|_{L^2(\mathbf{X})}^2 \\
&\ge 0
\end{align*}
This is a contradiction.
\end{proof}
Thus, the remaining task is to prove Theorem~\ref{thm:maindecomposition}.

\section{Proof of the ergodic-theoretic problem}
We now prove Theorem~\ref{thm:maindecomposition}. Our first proposition lets us discard the non-pretentious part whilst taking an average.
\begin{prop}\label{prop:reductiontopretentious}
Let $f\in L^\infty(\mathbf{X})$ have spectral measure supported on the aperiodic multiplicative functions. Then
$$\limsup_{L,r,K,Y,X} \|\mathbb{E}_{q_1 \in \Phi_r}\mathbb{E}_{q_2 \in \Phi_{r, K}}\mathbb{E}_{\substack{m, n \le X \\ m, n \sim w_L}} T_{P_i(M, N)/P_j(M, N)}f\|_{L^2(\mathbf{X})} = 0.$$
\end{prop}
\begin{proof}
To prove this, we require the following intermediate claim.
\begin{claim}
We may write
$$\mathbb{E}_{m, n \le X} \left|w_L(m, n) - \sum_{j = 1}^M a_j m^{i u_j}n^{-iu_j}\right| \le \varepsilon$$
where $|a_j| \le 1$, $u_j$ are real, and $M$ depends only on $L, \varepsilon$.
\end{claim}
\begin{proof}
Write $w_L(m, n) = \exp(\log(m/n))1_{[L, 2L]}$. The function $g(x) = e^x \chi(\log(m/n))(x)$ where $\chi$ is a smooth cutoff and equals one on $[L, 2L]$ can be rescaled into a piecewise Lipschitz function on $\frac{1}{\log(X)}\mathbb{Z}/\mathbb{Z}$. By a standard Fourier approximation argument applied to $g_{\mathrm{smooth}}$, we may write
$$g(x) =  \sum_{j = 1}^M a_je(iu_j x) + g_{\mathrm{sml}} + O(\varepsilon')$$
where $M$ depends only on $\varepsilon'$ and $L$.

By substituting $x = \log(m/n)$, we obtain
$$w_L(m, n) = \sum_{j = 1}^M a_j m^{iu_j}n^{-iu_j} + g_{\mathrm{sml}}(\log(m/n)) + O(\varepsilon').$$
This gives the desired decomposition.
\end{proof}
Now, since $\sum_{m,n\le X}w_L(m,n)\asymp_L X^2$, letting $\varepsilon\to0$, it suffices to show that for each $t$ we have
$$\limsup_{r, K, Y, X} \|\mathbb{E}_{q_1 \in \Phi_r} \mathbb{E}_{q_2 \in \Phi_{r, K}} \mathbb{E}_{m, n \le X} 1_{\log(m/n) \in [L, 2L]} m^{it}n^{-it}T_{P_i(M, N)/P_j(M, N)}f\| = 0.$$
In fact, we shall fix $q_1, r, q_2, K$, and $Y$. Applying the spectral theorem, it suffices to show that
$$\limsup_{X \to \infty} \int |\mathbb{E}_{m, n \le X} 1_{\log(m/n) \in [L, 2L]} m^{it}n^{- it} g(P_i(Qm+1,Qn+v))\overline{g(P_j(Qm+1,Qn+v))}|^2 d\sigma_f(g) = 0.$$
Note that we may write
$$\mathbb{E}_{m, n \le X} 1_{\log(m/n) \in [L, 2L]} m^{it}n^{- it} g(P_i(Qm+1,Qn+v))\overline{g(P_j(Qm+1,Qn+v))} $$
$$= \frac{1}{X^2}\sum_{\substack{m, n \le Q(X + 1) \\ m \equiv v \pmod{Q} \\n \equiv 1 \pmod{Q}}} \left(\frac{m-v}{Q}\right)^{it}\left(\frac{n-1}{Q}\right)^{-it} 1_{\log(m/n) \in [L, 2L]} g(P_i(m, n))\overline{g(P_j(m, n))} + O(Q^2/X^2).$$
Via a Taylor expansion, we have
$$\left(\frac{m-v}{Q}\right)^{it}\left(\frac{n-1}{Q}\right)^{-it} = \left(\frac{m-v}{n-1}\right)^{it} = \left(\frac{m}{n}\right)^{it} + O_Q\left(|t|\left(\frac{1}{m}+\frac{1}{n}\right)\right).$$
After normalization, the latter error contributes $O_Q(\log X/X)$ and therefore tends to zero. Thus, we obtain
$$\frac{1}{X^2}\sum_{\substack{m, n \le Q(X + 1) \\ m \equiv v \pmod{Q} \\n \equiv 1 \pmod{Q}}} 1_{\log(m/n) \in [L, 2L]} m^{it}n^{-it} g(P_i(m, n))\overline{g(P_j(m, n))} + O_Q(1/X^2).$$
Now Fourier-expand the $1_{n \equiv 1 \pmod{Q}}$ into Dirichlet characters. (Note we can't do the same for $m$ since $v$ is not relatively prime to $Q$.) Using Fatou's lemma, it suffices to prove the corresponding cancellation estimate for each resulting periodic factor; equivalently, one is reduced to averages of the form
$$\mathbb{E}_{m,n\le X} \chi(n)1_{m \equiv v \pmod{Q}}m^{it}n^{-it}g(P_i(m,n))\overline{g(P_j(m,n))},$$
where $\chi$ and $\chi'$ are characters modulo a divisor of $Q$. 
The result follows from the Frantzikinakis--Host's results \cite[Theorem 2.5 and Theorem 9.7]{Fra-Host-structure-multiplicative}. For instance, if one of $i$ or $j$ is $2$, then there is $(m, n) \mapsto m - n$ among the linear forms, so the theorem applies verbatim. If $i = 1$ and $j = 3$ (or vice versa), $g(P_i(m, n))$ factors as $g(2)g(m)g(n)$. There is $(m, n) \mapsto n$ among the linear forms, and since $g$ is non-pretentious, so is $n \mapsto \chi(n)n^{-it}g(n)$. (Hence, it only mattered that we Fourier expanded one of the indicators and we did not need to Fourier expand both of them for the argument to work.) This proves the proposition.
\end{proof}

We are now ready to prove Theorem~\ref{thm:maindecomposition}.
\begin{proof}[Proof of Theorem~\ref{thm:maindecomposition}.]
Applying Proposition~\ref{prop:reductiontopretentious}, we may replace $f$ with its conditional expectation to the pretentious factor. Now fix $ij$ and decompose $f = \mathcal{P}_{ij}f + f_{\mathrm{comp}}$. Our objective is to show that
\begin{equation}\label{eq:structuredcontrol}
\limsup_{L, r, K, Y, X} \|\mathbb{E}_{q_1 \in \Phi_r}\mathbb{E}_{q_2 \in \Phi_{r, K}}\mathbb{E}_{\substack{m, n \le X \\ m, n \sim w_L}} T_{P_i(M, N)/P_j(M, N)}\mathcal{P}_{ij}f - \mathcal{P}_{ij}f\|_{L^2(\mathbf{X})} = 0    
\end{equation}
and
\begin{equation}\label{eq:randomcontrol}
\limsup_{L, r, K, Y, X} \|\mathbb{E}_{q_1 \in \Phi_r}\mathbb{E}_{q_2 \in \Phi_{r, K}}\mathbb{E}_{\substack{m, n \le X \\ m, n \sim w_L}} T_{P_i(M, N)/P_j(M, N)}f_{\mathrm{comp}}\|_{L^2(\mathbf{X})} = 0.    
\end{equation}

\textbf{Step 1: Proving \eqref{eq:structuredcontrol}.}
The estimate \eqref{eq:structuredcontrol} is immediate for $ij \in \{12, 13\}$. To prove that it holds for $ij = 23$, the spectral theorem reduces the claim to
$$\int_{\mathcal{M}} |\mathbb{E}_{q_1 \in \Phi_r}\mathbb{E}_{q_2 \in \Phi_{r, K}}\mathbb{E}_{\substack{m, n \le X \\ m, n \sim w_L}} g(P_2(Qm + 1, Qn + v))\overline{g}(P_3(Qm + 1, Qn + v)) - 1|^2 d\sigma_{\mathcal{P}_{23}f}(g).$$
Here, we may assume that $g(p) = \chi(p)p^{it_g}$ for all but finitely many $p$. Thus, taking $K$ and $r$ sufficiently large, we may assume that $g(P_2(Qm + 1, Qn + v))\overline{g}(P_3(Qm + 1, Qn + v)) = P_2(Qm + 1, Qn + v)^{it_g}P_3(Qm + 1, Qn + v)^{-it_g}$. \eqref{eq:structuredcontrol} thus follows from the following claim.
\begin{claim}
$$\limsup_{L\to\infty}\limsup_{X\to\infty} \left|\mathbb{E}_{\substack{m, n \le X \\ m, n \sim w_L}} \left(\frac{P_2(M,N)}{P_3(M,N)}\right)^{it}-1\right|=0.$$
\end{claim}
\begin{proof}
Since $\log(m/n) \in [L, 2L]$ within the domain of averaging, we have $e^{-2L}\le n/m\le e^{-L}$. Since $n\ge1$, this also gives $1/m\le e^{-L}$. It follows that
$$\frac{N}{M} = \frac{Qn+v}{Qm+1} \le \frac{Qn + Q}{Qm} = \frac{n}{m} + \frac{1}{m} \le 2e^{-L}.$$
Thus, letting $r = N/M$, we have
$$\frac{P_2(M, N)}{P_3(M, N)} = \frac{1 - r^2}{1 + r^2}$$
so
$$\left|\log\frac{P_2(M, N)}{P_3(M, N)}\right| \ll r^2 \ll e^{-2L}.$$
It thus follows that
$$|(P_2(M, N)/P_3(M, N))^{it} - 1| \le |t| \left|\log\frac{P_2(M, N)}{P_3(M, N)}\right| \ll |t|e^{-2L}.$$
Hence the claim.
\end{proof}
\textbf{Step 2: Proving \eqref{eq:randomcontrol}.} By the spectral theorem, it suffices to show the following.
$$\limsup_{L, r, K, Y, X} \int |\mathbb{E}_{q_1 \in \Phi_r}\mathbb{E}_{q_2 \in \Phi_{r, K}}\mathbb{E}_{\substack{m, n \le X \\ m, n \sim w_L}} g(P_i(M, N)/P_j(M, N))|^2 d\sigma_{f_{\mathrm{comp}}}(g)= 0.$$
We show that for all $g$ considered in this integral that
$$\limsup_{L, r, K, Y, X} |\mathbb{E}_{q_1 \in \Phi_r}\mathbb{E}_{q_2 \in \Phi_{r, K}}\mathbb{E}_{\substack{m, n \le X \\ m, n \sim w_L}} g(P_i(M, N)/P_j(M, N))| = 0.$$
The proof naturally splits into several cases. 

\textbf{Case 1:} Suppose that $g$ has finite pretentious distance to $n\mapsto\chi(n)n^{it}$ and that one of the following occurs.
\begin{itemize}
    \item $g$ is not identically equal to $n\mapsto\chi(n)n^{it}$ and $ij\in\{12,13\}$;
    \item $g(p)\neq\chi(p)p^{it}$ for infinitely many primes $p$ and $ij = 23$;
    \item $\chi$ is nontrivial and $ij\in\{12,13\}$.
\end{itemize}
At this point, we assume that $L$ is fixed. By \eqref{eq:gcd-normalization}, we may factor $P_1(M, N) = q_1R_1(m, n)$ and $P_2(M, N) = q_2R_2(m, n)$ with both $R_1$ and $R_2$ relatively prime to $Q$. 

Applying the concentration estimates \cite[Proposition 2.1, Proposition 2.2]{pret_pyth} (see also \cite{Fra-Klu-Mor}), and using \eqref{eq:crt-product-congruences} and the fact that, for fixed $L$, the normalized $w_L$-average has density $O_L(1)$ with respect to the uniform box average, we may approximate $g$ with $\chi(n)n^{it}$ with the corresponding pretentious model, with the error controlled by the cited estimates. Specifically, the main terms are obtained by replacing $g(P_1(M, N))$ with 
$$g(q_1)\chi(P_1(M, N)/q_1)(P_1(M, N)/q_1)^{it} \exp(F_X(g, Y)),$$
we replace $g(P_2(M, N))$ with 
$$g(q_2)\chi(P_2(M, N)/q_2)(P_2(M, N)/q_2)^{it} \exp(F_X(g, Y)),$$
and we replace $g(P_3(M, N))$ with 
$$\chi(P_3(M, N))(P_3(M, N))^{it} \exp(G_{P_3, X}(g, Y)).$$
Now, if $ij = 12$, the extra exponential factor becomes $\exp(\mathrm{Re}(F_X(f, Y)))$, and via a Taylor expansion, we have that
$$\left(\frac{P_1(M, N)}{P_2(M, N)}\right)^{it} = \left(\frac{2mn}{m^2 - n^2}\right)^{it} + O_Q\left(\frac{1}{m}\right).$$
The latter error term is negligible, so we are left with
$$g(q_1)\overline{g}(q_2)q_1^{-it}q_2^{it}\chi(q_2)\left(\frac{2mn}{m^2 - n^2}\right)^{it}.$$
Averaging this in $q_1$ gives that
$$\limsup_{r \to \infty} \left|\mathbb{E}_{q_1 \in \Phi_r}g(q_1)\overline{g}(q_2)q_1^{-it}q_2^{it}\overline{\chi(q_1)}\chi(q_2)\left(\frac{2mn}{m^2 - n^2}\right)^{it}\right| = 0$$
since $g \neq \chi(n)n^{it}$ or $\chi$ is nontrivial.

A similar argument works for $ij = 13$. The only additional point is that $\exp(F_X(g, Y))$ and $\exp(G_X(g, Y))$ are uniformly bounded because $g$ is pretentious. Now if $ij = 23$, we instead have 
$$\limsup_{r \to \infty} \limsup_{K \to \infty} \left|\mathbb{E}_{q_2 \in \Phi_{r, K}}g(q_2)q_2^{-it}\overline{\chi(q_2)}\left(\frac{m^2 - n^2}{m^2 + n^2}\right)^{it}\exp(F_X(f, Y)) \exp(\overline{G_{P_3, X}(f, Y)})\right| = 0$$
as long as $g$ is not equal to $p \mapsto p^{it} \chi(p)$ for infinitely many primes $p$. 

\textbf{Case 2: Handling the archimedean case for $ij \in \{12, 13\}$.}
It remains to treat the nontrivial Archimedean characters $g(n)=n^{it}$ with $t\neq0$. (The case $t=0$ is the invariant component already contained in $\mathcal{P}_{ij}f$.) This follows from the next claim.
\begin{claim}
For every fixed $t\neq0$,
$$\limsup_{L\to\infty}\limsup_{X\to\infty} |\mathbb{E}_{\substack{m, n \le X \\ m, n \sim w_L}} (P_i(M, N)/P_j(M, N))^{it}| = 0.$$
\end{claim}
\begin{proof}

We estimate
$$\frac{N}{M} - \frac{n}{m} = \frac{Qn + v}{Qm + 1} - \frac{n}{m} = \frac{vm - n}{m(Qm + 1)}.$$
Since $0 \le v < Q$ and $n < m$ on the support of $w_L$, it follows that $|vm - n| \le (Q + 1)m$. Hence
$$\frac{N}{M} - \frac{n}{m} = O(1/m).$$
Now letting $R = N/M$, we have
$$\frac{P_1(M, N)}{P_2(M, N)} = \frac{2R}{1 - R^2}, \quad \frac{P_1(M, N)}{P_3(M, N)} = \frac{2R}{1 + R^2}.$$
Letting $r = n/m$, we have
$$\frac{P_1(m, n)}{P_2(m, n)} = \frac{2r}{1 - r^2}, \quad \frac{P_1(m, n)}{P_3(m, n)} = \frac{2r}{1 + r^2}.$$
For fixed $Q,v,L$, the affine shifts therefore disappear in the weighted Riemann-sum limit as $X\to\infty$. Equivalently, the limiting weighted average is the same as the one obtained from the homogeneous pair $(m,n)$. Now let $u=\log(m/n)$. Note that
$$\frac{P_1(m, n)}{P_2(m, n)} = \frac{2e^u}{e^{2u} - 1}, \quad \frac{P_1(m, n)}{P_3(m, n)} = \frac{2e^u}{1 + e^{2u}}.$$
For fixed $L$ and bounded Riemann-integrable $G$, the weighted Riemann-sum calculation gives
$$\lim_{X \to \infty} \mathbb{E}_{\substack{m, n \le X \\ m, n \sim w_L}} G(\log(m/n)) = \frac{1}{L} \int_L^{2L} G(t) dt.$$
It remains to show that
$$\frac{1}{L} \int_{L}^{2L} \left(\frac{2e^u}{e^{2u} - 1}\right)^{it} du = o_{L \to \infty}(1), \quad \frac{1}{L} \int_{L}^{2L} \left(\frac{2e^u}{e^{2u} + 1}\right)^{it} du = o_{L \to \infty}(1).$$
Note that 
$$\log\left(\frac{2e^u}{e^{2u} - 1}\right) = \log(2) - u - \log(1 - e^{-2u}) = \log(2) - u + O(e^{-2u}).$$
It follows that
$$\left(\frac{2e^u}{e^{2u} - 1}\right)^{it} = e^{it(\log(2) - u)} + O(|t|e^{-2L}).$$
Substituting this estimate gives
$$e^{it\log(2)}\frac{1}{L} \int_{L}^{2L} e^{-itu} du + O(|t|e^{-2L}) = O\left(\frac{1}{|t|L} + |t|e^{-2L}\right).$$
Similarly, we have
$$\log\left(\frac{2e^u}{e^{2u} + 1}\right) = \log(2) - u - \log(1 + e^{-2u}) = \log(2) - u + O(e^{-2u})$$
so the second integral is also $O\left(\frac{1}{|t|L} + |t|e^{-2L}\right)$. This completes the proof.
\end{proof}
This proves \eqref{eq:randomcontrol} and thus Theorem~\ref{thm:maindecomposition}.

\end{proof}

\section{AI acknowledgement and methodology}\label{sec:aiacknowledgement}
As remarked in the abstract, this proof was essentially found using ChatGPT 6.0 Astra Ultra on the Codex app \cite{ChatPythagorean}. The author subsequently rewrote the argument, with additional assistance from ChatGPT\footnote{The author found querying ChatGPT Instant and Extra High particularly useful for this purpose}, to supply more detail and introduce the ergodic-theoretic formalism. We describe here how the proof was obtained. The author started with the following prompt: \\\\
``Suppose the natural numbers are colored a finite number of colors. Prove that there exists $x, y, z$ of the same color such that $x^2 + y^2 = z^2$. Use maximum effort. Spend 1 day on this. Build an AI civilization to dispatch this problem if you have to. I attached a non-public manuscript recording some partial progress I have on this problem. It may or may not be helpful. Spend one agent to investigate this route though 3 other agents to investigate different routes.''\\\\
The author then attached an unpublished manuscript related to this problem.\footnote{As the unpublished manuscript is a work in progress, we do not make it available here.} A few minutes later, the author prompted it, ``Spend one subagent looking for a disproof.'' After ten hours, it had not proved or disproved the statement. The author subsequently asked it for a new proof of the two-color case, which produced the argument developed here. 
\bibliographystyle{amsplain}
\bibliography{main}

\end{document}